\documentclass[12pt]{amsart}
\usepackage{amsmath}
\usepackage{graphicx}
\usepackage{hyperref}
\usepackage[latin1]{inputenc}
\usepackage{mathtools}
\usepackage{amsfonts}
\usepackage{amssymb}
\usepackage[T1]{fontenc}
\usepackage{amsthm}
\usepackage{fullpage}
\usepackage{enumitem}

\newtheorem{theorem}{Theorem}[section]

\newtheorem{lemma}[theorem]{Lemma}

\theoremstyle{definition}
\newtheorem{definition}{Definition}[section]

\theoremstyle{remark}

\numberwithin{equation}{section}

\newcommand{\F}{\mathbb{F}_q}
\newcommand{\Fm}{\mathbb{F}_{q^m}}
\newcommand{\R}{\mathfrak{R}}
\newcommand{\Q}{\mathfrak{Q}}
\newcommand{\Nm}{\mathfrak{N}}
\newcommand{\M}{\mathfrak{M}}
\newcommand{\vt}{\vartheta}

\DeclareMathOperator{\Tr}{Tr}
\DeclareMathOperator{\N}{N}

\title{Primitive normal Values of rational functions with one prescribed norm and trace over finite fields}


\author{Arpan Chandra Mazumder}
\address{Department of Mathematical Sciences, Tezpur University, Tezpur, Assam, 784028, India}
\email{arpan10@tezu.ernet.in}

\author{Dhiren Kumar Basnet}
\address{Department of Mathematical Sciences, Tezpur University, Tezpur, Assam, 784028, India}
\email{dbasnet@tezu.ernet.in}

\thanks{The first author is supported by DST INSPIRE Fellowship, under grant no. DST/INSPIRE Fellowship/2021/IF210206. }

\begin{document}
	%\vspace{.3cm}
	\begin{abstract}
		Let $q, n, m \in \mathbb{N}$ be such that $q$ is a prime power and $a, b \in \F$. In this article we establish a sufficient condition for the existence of a primitive normal pair $(\alpha, f(\alpha)) \in \Fm$ over $\F$ with a prescribed primitive norm $a$ and a non-zero trace $b$ over $\F$ of $\alpha$, where $f(x) \in \Fm(x)$ is a rational function of degree sum $n$ with some minor restrictions. Furthermore, for $q=7^k$, $m \geq 7$ and rational functions with numerator and denominator being linear, we explicitly find at most 6 fields in which the desired pair may not exist. 
	\end{abstract}
	
	\maketitle
	
	\section{Introduction}
	
	Throughout the article, let $q$ be a prime power and $m$ be a positive integer. Denote by $\F$ the finite field of order $q$ and by $\Fm$ the extension field of $\F$ of degree $m$. The multiplicative group $\Fm^*$ is cyclic and a generator of this group is called a primitive element of $\Fm^*$. An element $\alpha \in \Fm$ is said to be normal over $\F$ if the set of all its conjugates with respect to $\F$, that is, the set $\{\alpha, \alpha^{q}, \ldots , \alpha^{q^{m-1}}\}$  forms a basis of $\Fm$ over $\F$. An element $\alpha \in \Fm$ is said to be a primitive normal if it is both primitive and normal over $\F$. 
	
	Let $\alpha \in \Fm$, then $\alpha, \alpha^{q}, \ldots , \alpha^{q^{m-1}}$ are called the conjugates of $\alpha$ over $\F$.The trace of an element $\alpha \in \Fm$ over $\F$, denoted by $\Tr_{\mathbb{F}_{q^m}/\mathbb{F}_{q}}(\alpha)$, is the sum of all conjugates of $\alpha$ with respect to $\F$, that is, $\Tr_{\mathbb{F}_{q^m}/\mathbb{F}_{q}}(\alpha)= \alpha+\alpha^{q}+\ldots+\alpha^{q^{m-1}}$. The norm of an element $\alpha \in \Fm$ over $\F$, denoted by $\N_{\mathbb{F}_{q^m}/\mathbb{F}_{q}}(\alpha)$, is the product of all the conjugates of $\alpha$ with respect to $\F$, that is, $\N_{\mathbb{F}_{q^m}/\mathbb{F}_{q}}(\alpha)= \alpha. \alpha^{q}\ldots\alpha^{q^{m-1}}=\alpha^\frac{{q^m-1}}{q-1}$. We note that the trace of a normal element is a non-zero element and the norm of a primitive element is a primitive element. 
	
	Primitive elements, besides their theoretical interest, have various applications, including cryptographic schemes\cite{difhell} such as Diffie-Hellmen key exchange and Elgamel encryption schemes and the construction of Costas arrays\cite{costas}, which are used in sonar and radar technology. Normal elements hold computational advantages for finite field arithmetic and are therefore used in many software and hardware implementations in coding theory and cryptography.
	
	For any $\alpha, \beta \in \Fm$, we call a pair $(\alpha, \beta)$ to be a primitive normal pair if both $\alpha$ and $\beta$ are primitive and normal elements over $\F$. The existence of primitive normal pairs $(\alpha, f(\alpha))$, where $f(x)$ is a function over $\Fm$ have been studied extensively and is a topic of great research interest. A natural question is to ask whether there exists a primitive normal pair of elements with prescribed trace and norm.
	
	In 1999 and 2000, in a series of two papers \cite{cohach1, cohach2} Cohen and Hachenberger proved the existence of primitive and normal elements of prescribed trace and norm over finite fields. In 2001, Chou and Cohen \cite{chou} proved the existence of an element $\alpha$ such that both $\alpha$ and $\alpha^{-1}$ have zero traces over $\F$. In 2018, Gupta, Sharma and Cohen \cite{gupta1}, proved the existence of a primitive pair $(\alpha, \alpha+\alpha^{-1})$ with $\Tr_{\mathbb{F}_{q^m}/\mathbb{F}_{q}}(\alpha)=a$, for any prescribed $a \in \F$ and for all $q$ and $m \geq 5$. Recently, Sharma, Rani and Tiwari \cite{avnish2} proved the existence of a primitive normal pair $(\alpha,f(\alpha))$ where $f(x) \in \Fm[x]$ is a polynomial with minor restrictions such that  $\N_{\mathbb{F}_{q^m}/\mathbb{F}_{q}}(\alpha)=a$ and $\Tr_{\mathbb{F}_{q^m}/\mathbb{F}_{q}}(\alpha)=b$.
	
	We impose a minor constraint on the form of the rational function $f \in \Fm(x)$ so as to avoid some exceptional cases.  We consider rational functions $f \in \Fm(x)$ which belong to the set $\R_n$ as defined below.
	
	\begin{definition}
		The set $\R_n$ is the collection of rational functions $f \in \Fm(x)$ with simplest form $f_1/f_2$, where $n_1, n_2$ are the degrees of $f_1, f_2$ respectively, with $n_1+n_2=n$ i.e. degree sum of $f$ is $n$, such that
		
		\begin{enumerate}[label=(\roman*)]
			\item $f \neq cx^ig^d$ for some $c \in \Fm$, $i \in \mathbb{Z}$, $g \in \Fm(x)$ and $d(> 1)$ divides $q^m-1$.
			\item \label{part(ii)} $f_2$ is monic such that $f_2\neq g^{q^m}$ for any $g\in\Fm[x]$.%there exists at least one monic irreducible factor $g$ of $f_2 \in \Fm[x]$ with multiplicity $t$ such that $q^m \nmid t$. 
		\end{enumerate}
	\end{definition}
	
	We note that if $f(x) = g(x)^t$ for some $g \in \Fm(x)$ and $t>1$ which divides $q^m-1$, then $f(\alpha)$ (for any primitive element $\alpha$) is necessarily a $t^\text{th}$ power and hence it cannot be primitive. Furthermore, if $f(x) = cxg(x)^2$ where $c$ a non-square in $\F$. Then, if $\alpha$ is primitive, $f(\alpha)$ is a square and so not necessarily primitive. Also, part \ref{part(ii)} of the above definition forces that $n_2 \geq 1$. For the case $n_2 = 0$, i.e. for polynomials $f \in \Fm[x]$, the existence of a primitive normal element $\alpha \in \Fm$ over $\F$ such that $f(\alpha)$ is also a primitive normal element in $\Fm$ over $\F$ with prescribed $a \in \F$ has been studied in \cite{avnish2}.
	
	In this article we aim to classify finite fields $\Fm$ for which there exists a primitive normal pair $(\alpha, f(\alpha))$ in $\Fm$ over $\F$ such that $\N_{{\Fm}/{\F}}(\alpha)= a$ and $\Tr_{{\Fm}/{\F}}(\alpha)= b$ for any prescribed primitive norm $a$ and a non-zero trace $b$ over $\F$, where $f(x) \in \R_n$. To be precise, in Section 3, we obtain a sufficient condition for the existence of desired primitive normal pairs. In Section 4, we use an "additive-multiplicative" sieve to weaken the sufficient condition for more efficient results. In Section 5, we study the application of our results over fields of chracteristic $7$ and for rational functions whose degrees of  numerator and denominator are $1$ and proved the following:
	
	\begin{theorem}\label{comp}
		Let $q=7^k$, $m\geq7$ and $f \in \R_n$ such that $(n_1, n_2)=(1, 1)$. Then we have that, for any prescribed primitive norm $a$ and a non-zero trace $b$ over $\F$, there exists an element $\alpha \in \Fm$ with $\N_{{\Fm}/{\F}}(\alpha)= a$ and $\Tr_{{\Fm}/{\F}}(\alpha)= b$, for which $(\alpha, f(\alpha))$ is  a primitive normal pair  in $\mathbb{F}_{q^m}$ over $\mathbb{F}_{q}$, except possibly the pairs $(7, 7)$, $(7, 8)$, $(7, 9)$, $(7, 10)$, $(7, 12)$ and  $(7, 18)$.
	\end{theorem}
	
	We introduce some notations and conventions which play an important role in this article. For $q=p^k$, we denote by $\mathbb{F}$ the algebraic closure of $\mathbb{F}_p$. Let $\Q(n_1, n_2)$ denote the set of pairs $(q, m)$ such that, for any prescribed primitive norm $a$ and a non-zero trace $b$ over $\F$ and $f\in\R_n$, there exists an element $\alpha \in \Fm$ with $\N_{{\Fm}/{\F}}(\alpha)= a$ and $\Tr_{{\Fm}/{\F}}(\alpha)= b$, for which $(\alpha, f(\alpha))$ is  a primitive normal pair  in $\mathbb{F}_{q^m}$ over $\mathbb{F}_{q}$. For each $n \in \mathbb{N}$, we denote by $\omega(n)$ and $W(n)$, the number prime divisors of $n$ and the number of square-free divisors of $n$ respectively. Also for $f(x) \in \F[x]$, we denote by $\omega(f)$ and $W(f)$, the number of monic irreducible $\F$-divisors of $f$ and the number of square-free $\F$-divisors of $f$ respectively.
	
	All non-trivial computations wherever needed in this article are carried out using SageMath \cite{sagesage}. 
	 
	\section{Preliminaries}
	
	In this section, we recall some definitions and results and provide some preliminary notations which are used to prove the main results of this article.
	
	In our work, the characteristic functions of primitive and normal elements play an important role. To represent those functions, the idea of character of finite abelian group is necessary.
	
	Let $\mathbb{G}$ be a finite abelian group. A character $\chi$ of the group $\mathbb{G}$ is a homomorphism from $\mathbb{G}$ to $\mathbb{S}^1$, where $\mathbb{S}^1:= \lbrace z\in \mathbb{C}: |z| = 1 \rbrace$ is the multiplicative group of complex numbers of unit modulus, i.e. $\chi(a_1a_2)= \chi(a_1)\chi(a_2)$, for all $a_1, a_2 \in \mathbb{G}$.
	
	The characters of $\mathbb{G}$ form a group under multiplication defined by $(\chi_\alpha\chi_\beta)(a)= \chi_\alpha(a)\chi_\beta(a)$, called the dual group or character group of $\mathbb{G}$ denoted by $\widehat{\mathbb{G}}$, which in fact is isomorphic to the group $\mathbb{G}$. If the order of an element (i.e. a character) of the group $\widehat{\mathbb{G}}$ is $d$, then  characters of order $d$ is denoted by $(d)$. Further, since $\widehat{\mathbb{F}^*_{q^m}} \cong \mathbb{F}^*_{q^m}$, $\widehat{\mathbb{F}^*_{q^m}}$ is cyclic and for any divisor $d$ of $q^m-1$ there are exactly $\phi(d)$ characters of order $d$ in $\widehat{\mathbb{F}^*_{q^m}}$. The special character $\chi_0: \mathbb{G} \rightarrow \mathbb{S}^1$ defined as  $\chi_0(a) = 1$, for all $a \in \mathbb{G}$ is called the trivial character of $\mathbb{G}$.
	
	In a finite field $\mathbb{F}_{q^m}$ there are two group structures, one is the additive group $\mathbb{F}_{q^m}$ and the other is the multiplicative group $\mathbb{F}^*_{q^m}$. Therefore we have two types of characters pertaining to these two group structures, one is the additive character for $\mathbb{F}_{q^m}$ denoted by $\psi$ and the other one is the multiplicative character for $\mathbb{F}^*_{q^m}$ denoted by $\chi$. The multiplicative characters associated with $\mathbb{F}^*_{q^m}$ are extended from $\mathbb{F}^*_{q^m}$ to $\mathbb{F}_{q^m}$ by the rule
	\[ \chi(0)=\begin{cases}
		0,& \text{if } \chi\neq\chi_0,\\
		1,& \text{if } \chi=\chi_0. 
	\end{cases} \]
	
	Let $e\mid q^m-1$, then an element $\alpha \in \mathbb{F}^{*}_{q^m}$ is called $e$-free, if $d\mid e$ and $\alpha = y^d$, for some $y \in \mathbb{F}^{*}_{q^m}$ imply $d=1$. It is clear from the definition that an element $\alpha \in \mathbb{F}^{*}_{q^m}$ is primitive if and only if it is a $(q^m-1)$-free element.  The $p$-free part of a natural number $r$ is denoted by $r_0$, where $r= r_0 p^k$ such that $\mathrm{gcd}(r_0, p) = 1$. It is clear from the definition that $q^m-1$ can be freely replaced by its $p$-free part.
	
	For any $e\mid q^m-1$, the characteristic function of $e$-free elements of $\mathbb{F}^{*}_{q^m}$ is defined as follows:
	\begin{equation}\label{e-free ch}
		\rho_e: \mathbb{F}^{*}_{q^m}\rightarrow \{0,1\}; \alpha \mapsto \theta(e) \sum_{d\mid e} \left( \frac{\mu(d)}{\phi(d)} \sum_{(d)} \chi_d(\alpha) \right),
	\end{equation}
	where $\theta(e):= \phi(e)/e$, $\mu$ is the M\"obius function and $\chi_d$ stands for any character of $\widehat{\mathbb{F}_{q^m}^*}$ of order $d$.
	
	The additive group of $\mathbb{F}_{q^m}$ is an $\mathbb{F}_{q}[x]$-module under the rule $$f\circ\alpha=\sum_{i=0}^n a_i\alpha^{q^i},$$ for $\alpha\in \mathbb{F}_{q^m}$, where $$f(x)= \sum_{i=0}^na_ix^i\thinspace \in \mathbb{F}_{q}[x]$$.\\ For $\alpha \in \mathbb{F}_{q^m}$, the $\mathbb{F}_q$-order of $\alpha$ is the monic $\mathbb{F}_q$-divisor $g$ of $x^m-1$ of minimal degree such that $g \circ\alpha=0$.
	
	The \emph{$\mathbb{F}_q$-order} of an additive character $\chi_g \in \widehat{\mathbb{F}_{q^m}}$ is the monic $\mathbb{F}_{q}$-divisor $g$ of $x^m-1$ of minimal degree such that $\chi_g \circ g$ is the trivial character of $\widehat{\mathbb{F}_{q^m}}$, where ($\chi_g \circ g)(\alpha)= \chi_g(g \circ \alpha)$ for any $\alpha \in \Fm$.
	
	For $g\mid x^m-1$, an element $\alpha\in \mathbb{F}_{q^m}$ is called $g$-free element if $\alpha = h \circ \beta$ for some $\beta \in \mathbb{F}_{q^m}$ and $h\mid g$ imply $h=1$. From the definition, it is obvious that an element $\alpha \in \mathbb{F}_{q^m}$ is normal if and only if it is $(x^m-1)$-free. It is clear that $x^m-1$ can be freely replaced by its $p$-radical $g_0:= x^{m_0}-1$, where $m_0$ is such that $m=m_0p^a$, here $a$ is a non-negative integer and $\mathrm{gcd} (m_0, p)=1$. 
	
	For any $g\mid x^m-1$, the characteristic function of $g$-free elements of $\mathbb{F}_{q^m}$ is defined as follows:
	\begin{equation}\label{g-free ch}
		\kappa_g : \mathbb{F}_{q^m}\rightarrow \{0, 1\}; \alpha \mapsto \Theta(g) \sum_{f\mid g} \left( \frac{\mu^\prime(f)}{\Phi(f)} \sum_{(f)} \psi_f(\alpha) \right),
	\end{equation}
	where $\Theta(g):= \Phi(g)/{q^{deg(g)}}$, $\psi_f$ stands for any character of $\widehat{\mathbb{F}_{q^m}}$ of $\mathbb{F}_{q}$-order $f$ and  $\mu^\prime$ is the analogue of the M\"obius function defined as follows:
	\[
	\mu^\prime(g)=\begin{cases}
		(-1)^s , & \text{if $g$ is the product of $s$ distinct irreducible monic polynomials}, \\
		0 , &\text{otherwise.}\\ 
	\end{cases}
	\]
	
	For each $b \in \mathbb{F}_{q}$, the characteristic function of the subset of $\mathbb{F}_{q^m}$ consisting of elements $\alpha$ with $\Tr_{\mathbb{F}_{q^m}/\mathbb{F}_{q}}(\alpha)=b$ is defined as follows: 
	\begin{equation}\label{Tr ch}
		\tau_b : \mathbb{F}_{q^m}\rightarrow \{0, 1\}; \alpha \mapsto \frac{1}{q}\sum_{\psi \in \widehat{\mathbb{F}_{q}}} \psi\left(\Tr_{\mathbb{F}_{q^m}/\mathbb{F}_{q}}(\alpha)-b \right).
	\end{equation}
	
	The additive character $\psi_1$ defined by $\psi_1(\beta) = e^{{2{\pi}i\text{Tr}(\beta)}/p}$, for all $\beta \in \mathbb{F}_{q}$, where Tr is the absolute trace function from $\mathbb{F}_{q}$ to $\mathbb{F}_{p}$,
	is called the \emph{canonical additive character} of $\mathbb{F}_{q}$ and every additive character $\psi_\beta$
	for $\beta \in \mathbb{F}_{q}$ can be expressed in terms of the canonical additive character $\psi_1$ as
	$\psi_\beta(\gamma)=\psi_1(\beta\gamma)$, for all $\gamma \in \mathbb{F}_{q}$.
	
	Thus we have that 
	\begin{align*}
		\tau_b(\alpha) & =  \frac{1}{q}\sum_{u \in \mathbb{F}_{q}} \psi_1\left(\text{Tr}_{\mathbb{F}_{q^m}/\mathbb{F}_{q}}(u\alpha)-ub \right) \\ %\nonumber\\
		&=  \frac{1}{q}\sum_{u \in \mathbb{F}_{q}} \widehat{\psi_1}(u\alpha)\psi_1(-ub),
	\end{align*}
	where $\widehat{\psi_1}$ is the lift of $\psi_1$, that is, the additive character of $\mathbb{F}_{q^m}$ defined by $\widehat{\psi_1}(\alpha) = \psi_1(\text{Tr}_{{\mathbb{F}_{q^m}}/\mathbb{F}_{q}}(\alpha))$. In particular, $\widehat{\psi_1}$ is the canonical additive character of $\mathbb{F}_{q^m}$.
	
	For each $a \in \F^*$ the characteristic function of the subset of $\mathbb{F}_{q^m}$ consisting of elements $\alpha$ with $\N_{\mathbb{F}_{q^m}/\mathbb{F}_{q}}(\alpha)=a$ is defined as follows: 
	\begin{equation}\label{Norm ch}
		\eta_a : \mathbb{F}_{q^m}\rightarrow \{0, 1\}; \alpha \mapsto \frac{1}{q-1}\sum_{\chi \in \widehat{\F^*}} \chi\left(\N_{\mathbb{F}_{q^m}/\mathbb{F}_{q}}(\alpha).{a}^{-1} \right).
	\end{equation}
	
	The multiplicative group $\F^*$ is cyclic, so for any character $\chi$ of $\F^*$ we can write $\chi(\alpha)=\chi_{q-1}^i(\alpha)$ where $i \in \{1, 2, \ldots, q-1 \}$ and $\chi_{q-1}$ is the multiplicative character of order $q-1$ in $\widehat{\F^*}$. In particular we have that,
	
	\begin{align*}
		\eta_b(\alpha) & = ~\frac{1}{q-1}\sum_{\chi \in \widehat{\F^*}} \chi\left(\N_{\mathbb{F}_{q^m}/\mathbb{F}_{q}}(\alpha).{a}^{-1} \right) \\
		& = ~\frac{1}{q-1} \sum_{i=1}^{q-1} \chi_{q-1}(a^{-i})\chi_{q-1}^i(\N_{\Fm/\F}(\alpha)) \\
		& = ~\frac{1}{q-1} \sum_{i=1}^{q-1} \chi_{q-1}(a^{-i})\widehat{\chi^i}(\alpha)
	\end{align*}
	
	where $\widehat{\chi}$ is the lift of $\chi_{q-1}$, that is, the multiplicative  character of $\Fm^*$ defined by $\widehat{\chi}(\alpha) = \chi_{q-1}(\N_{{\mathbb{F}_{q^m}}/\mathbb{F}_{q}}(\alpha))$. We observe that $\widehat{\chi}^{q-1}(\alpha)=1$. In fact, the order of $\widehat{\chi}$ is $q-1$. To show this, let the order of $\widehat{\chi}$ be $j<q-1$. Then for any $\alpha \in \Fm^*$, we have that $\widehat{\chi}^j(\alpha)=1$ which implies $\chi^j_{q-1}(\N_{\Fm/\F}(\alpha))=1$ which further implies $\chi^j_{q-1}(\gamma)=1$ for all $\gamma \in \F^*$, since norm is an onto function. That is, $(q-1)|j$, which is a contradiction. 
	
	The following results are due to Wan and Fu \cite{fu} and are crucial for proving our sufficient condition.
	
	\begin{lemma}\cite[Theorem~5.5]{fu}\label{2.1}
		Let $f(x) \in \Fm(x)$ be a rational function. Write $f(x)= \prod_{j=1}^k f_j(x)^{r_j}$, where $f_j(x) \in \mathbb{F}_{q^m}[x]$ are irreducible polynomials and $r_j$ are non zero integers. Let $\chi$ be a non-trivial multiplicative character of $\mathbb{F}_{q^m}$ of square-free order $d$ (a divisor of $q^m-1$). Suppose that $f(x)$ is not of the form $cg(x)^d$ for any rational function $g(x) \in \mathbb{F}_{q^m}(x)$ and $c \in \mathbb{F}^*_{q^m}$. Then we have
		\begin{equation*}
			\left| \sum_{\alpha \in \mathbb{F}_{q^m},f(\alpha)\neq 0, \infty} \chi(f(\alpha)) \right| \leq \left({\sum_{j=1}^{k} deg(f_j)}-1\right)q^{\frac{m}{2}}.
		\end{equation*}
	\end{lemma}
	
	\begin{lemma}\cite[Theorem~5.6]{fu}\label{2.2}
		Let $f(x),g(x) \in \mathbb{F}_{q^m}(x)$ be a rational functions. Write $f(x)= \prod_{j=1}^k f_j(x)^{n_j}$, where $f_j(x) \in \mathbb{F}_{q^m}[x]$ are irreducible polynomials and $n_j$ are non zero integers. Let $D_1=\sum_{j=1}^{k} deg(f_j)$, $D_2=max(deg(g), 0)$, $D_3$ be the degree of the denominator of $g(x)$, and $D_4$ be the sum of degrees of those irreducible polynomials dividing the denominator of $g$ but distinct from $f_j(x)(j = 1, 2, \ldots , k)$. Let $\chi$ be a multiplicative character of $\mathbb{F}_{q^m}$, and let $\psi$ be a non trivial additive character of $\mathbb{F}_{q^m}$. Suppose $g(x)$ is not of the form $r(x)^{q^m}-r(x)$ in $\mathbb{F}_{q^m}(x)$. Then we have the estimate
		\begin{equation*}
			\left| \sum_{\alpha \in \mathbb{F}_{q^m},f(\alpha)\neq 0,\infty;g(\alpha)\neq 0,\infty} \chi(f(\alpha))\psi(g(\alpha)) \right| \leq (D_1+D_2+D_3+D_4-1)q^{\frac{m}{2}}.
		\end{equation*}
	\end{lemma}
	
	The following result provides an upper bound on the square-free divisors of $x^m-1$.
	
	%\begin{lemma}\cite[Lemma~3.7]{hucz2}\label{2.3}
		%For any $\alpha \in \mathbb{N}$ and a positive real number $\nu$, $W(\alpha) \leq C. \alpha^{1/\nu}$, where $C = \prod_{i=1}^{r} \frac{2}{p_i^{1/\nu}}$ and $p_1, p_2, \dots, p_r$ are the primes less than equal to $2^\nu$ that divide $\alpha$.
	%\end{lemma}
	
	\begin{lemma}\cite[Lemma~2.9]{lens}\label{2.3}
		Let $q$ be a prime power and $m$ a positive integer. Then, we have $W(x^m-1) \leq 2^{\frac{1}{2}(m+\gcd(m,q-1))}$. In particular, $W(x^m-1) \leq 2^m$ and $W(x^m-1) = 2^m$ if and only if $m \mid (q-1)$. Furthermore, $W(x^m-1) \leq  2^{3m/4}$ if $m \nmid (q-1)$, since in this case, $\gcd(m, q-1) \leq \frac{m}{2}$.
	\end{lemma}

	\section{The Main Result}
	
	It is known to us that norm of a primitive element is also a primitive element. In particular, for any $\alpha \in \Fm$ which is $(q^m-1)$-free, its norm over $\F$ is $(q-1)$-free. The following lemma is a generalization of the above fact done by Sharma et al. in Lemma 3.1 of \cite{avnish2}.
	
	\begin{lemma}
		Let $l\mid(q^m-1), \sigma = \gcd(l, q-1)$, and $Q_l$ be the largest divisor of $l$ such that $\gcd(Q_l, \sigma) = 1$. Then an element $\alpha \in \Fm^*$ is $l$-free if and only if $\N_{\Fm/\F}(\alpha)$ is $\sigma$-free in
		$\F^*$ and $\alpha$ is $Q_l$-free.
	\end{lemma}
	
	Let $e_1, e_2$ be divisors of $q^m-1$, $\sigma = \gcd(e_1, q-1)$ and $Q_{e_1}$ be the largest divisor of $e_1$ such that $\gcd(Q_{e_1}, \sigma)=1$. From Lemma 3.1, if the norm of an element $\alpha$ is $\sigma$-free, then for $\alpha$ to be $e_1$-free, it suffices to show that it is $Q_{e_1}$-free. 
	
	Furthermore, we know that trace of a normal element is non-zero. In fact, if the trace of $\alpha$ is non-zero, then $\alpha$ is $(x-1)$-free. Let $g_1(x), g_2(x)$ be divisors of $x^m-1$. Thus, for $\alpha$ to be $g_1$-free, it suffices to show that it is $R_{g_1}$-free, where $R_{g_1}$ is the largest divisor of $g_1$ co-prime to $(x-1)$. 
	
	For any rational function $f(x) \in \R_n$, any $\sigma$-free element $a \in \F^*$ and any non-zero element $b \in \F$, let $\Nm_{f, a, b}(Q_{e_1}, e_2, R_{g_1}, g_2)$ denote the number of elements $\alpha \in \Fm$ such that $\alpha$ is $Q_{e_1}$ and $R_{g_1}$-free, $f(\alpha)$ is $e_2$ and $g_2$-free, and $\N_{\Fm/\F}(\alpha) = a$ and $\Tr_{\Fm/\F}(\alpha) = b$.
	
	In this section we present a sufficient condition for the existence of a primitive normal pair $(\alpha, f(\alpha))$ in $\Fm$ over $\F$ such that $\N_{{\Fm}/{\F}}(\alpha)= a$ and $\Tr_{{\Fm}/{\F}}(\alpha)= b$ for any prescribed primitive norm $a$ and a non-zero trace $b$ over $\F$, where $f(x) \in \R_n$. From the discussion above, it is equivalent to determining when $\Nm_{f, a, b}(Q_{e_1}, e_2, R_{g_1}, g_2)>0$. 
	
	Let $P$ be the set containing $0$, and the zeros and poles of $f$. From the definition we have, 
	
	$$\Nm_{f, a, b}(Q_{e_1}, e_2, R_{g_1}, g_2)= \sum_{\alpha\in\mathbb{F}_{q^m} \setminus P}\rho_{Q_{e_1}}(\alpha)\rho_{e_2}(f(\alpha))\kappa_{R_{g_1}}(\alpha)\kappa_{g_2}(f(\alpha))\eta_a(\alpha)\tau_b(\alpha)$$
	
	Then using the characteristic functions \ref{e-free ch}, \ref{g-free ch}, \ref{Tr ch} and \ref{Norm ch} we have the following expression,
	
		\begin{align}\label{cond3.1}
		\Nm_{f, a, b}(Q_{e_1}, e_2, R_{g_1}, g_2) =&  ~\frac{\theta(Q_{e_1})\theta(e_2)\Theta(R_{g_1})\Theta(g_2)}{q(q-1)}
		\sum_{\substack{d_1\mid Q_{e_1}, d_2\mid e_2 \\ h_1\mid R_{g_1}, h_2\mid g_2}}\frac{\mu(d_1)\mu(d_2)\mu'(h_1)\mu'(h_2)}{\phi(d_1)\phi(d_2)\Phi(h_1)\Phi(h_2)} \nonumber \\ 
		& ~\sum_{\substack{\chi_{d_1}, \chi_{d_2} \\ \psi_{h_1}, \psi_{h_2}}}
		\chi_{f, a, b}(d_1, d_2, h_1, h_2) \nonumber\\
		=& ~\vt\sum_{\substack{d_1\mid e_1,d_2\mid e_2 \\ h_1\mid g_1,h_2\mid g_2}} \frac{\mu(d_1)\mu(d_2)\mu'(h_1)\mu'(h_2)}{\phi(d_1)\phi(d_2)\Phi(h_1)\Phi(h_2)} \sum_{\substack{\chi_{d_1},\chi_{d_2} \\ \psi_{h_1},\psi_{h_2}}} \chi_{f, a, b}(d_1, d_2, h_1, h_2),
	\end{align}
	
	where $$\vt = \frac{\theta(Q_{e_1})\theta(e_2)\Theta(R_{g_1})\Theta(g_2)}{q(q-1)}$$ and 
	
    \begin{align*}
	   \chi_{f, a, b}(d_1, d_2, h_1, h_2)=& ~\sum_{\chi \in \widehat{\F^*}}\sum_{\psi \in \widehat{\F}}\sum_{\alpha\in\Fm\setminus P} \chi_{d_1}(\alpha).\chi_{d_2}(f(\alpha)).\psi_{g_1}(\alpha).\psi_{g_2}(f(\alpha)). \\ & ~\chi(\N_{\Fm/\F}(\alpha).{a^{-1}}).\psi(\Tr_{\Fm/\F}(\alpha)-b)\\
	   =& ~\sum_{i=1}^{q-1}\sum_{c \in \F} \chi_{q-1}(a^{-i})\psi_1(-cb)\sum_{\alpha\in\Fm\setminus P}\chi_{d_1}(\alpha).\chi_{d_2}(f(\alpha)).\psi_{g_1}(\alpha).\psi_{g_2}(f(\alpha)). \\ &
	   ~\chi_{q-1}^i(\N_{\Fm/\F}(\alpha)).\psi_1(c\Tr_{\Fm/\F}(\alpha)) \\
	   =& ~\sum_{i=1}^{q-1}\sum_{c \in \F} \chi_{q-1}(a^{-i})\psi_1(-cb)\sum_{\alpha\in\Fm\setminus P}\chi_{d_1}(\alpha).\chi_{d_2}(f(\alpha)).\psi_{g_1}(\alpha).\psi_{g_2}(f(\alpha)). \\ &
	   ~\widehat{\chi^i}(\alpha).\widehat{\psi_1}(c\alpha).
    \end{align*}
	
	where, as seen earlier, $\widehat{\chi}$ and $\widehat{\psi_1}$ are lifts of $\chi_{q-1}$ and $\psi_1$ to $\Fm^*$ and $\Fm$ respectively.
	
	Since, $\widehat{\mathbb{F}}^*_{q^m}$ is a cyclic group, we can write $\widehat{\chi}=\chi_{q^m-1}^{(q^m-1)/(q-1)}$ as the order of $\widehat{\chi}$ is ${q-1}$ as seen earlier. Also, we can write  $\chi_{d_i}(\alpha)=\chi_{q^m-1}(\alpha^{n_i})$ for some $n_i \in \{0, 1, 2, \ldots , q^m-2\}; i=1, 2$. Furthermore, there exist $u_1, u_2 \in \Fm$ such that $\psi_{g_i}(\alpha)=\widehat{\psi}_1(u_i\alpha)$, for $i = 1, 2$, where $\widehat{\psi}_1$ is the canonical additive character of $\Fm$. Then we have the following expression,
	
	\begin{align}
		\chi_{f, a, b}(d_1, d_2, h_1, h_2)= &  ~\sum_{i=1}^{q-1}\sum_{c \in \F} \chi_{q-1}(a^{-i})\psi_1(-cb)\sum_{\alpha\in\Fm\setminus P}\chi_{q^m-1}(\alpha^{n_1+\frac{(q^m-1)i}{q-1}}f(\alpha)^{n_2}) \nonumber \\ & ~\widehat{\psi_1}((u_1+c)\alpha+u_2f(\alpha)) \\
		=& ~\sum_{i=1}^{q-1}\sum_{c \in \F} \chi_{q-1}(a^{-i})\psi_1(-cb)\sum_{\alpha\in\Fm\setminus P}\chi_{q^m-1}(H_i(\alpha))\widehat{\psi_1}(T_c(\alpha)) \nonumber,
	\end{align}
	
	where $H_i(x)= x^{n_1+\frac{(q^m-1)i}{q-1}}f(x)^{n_2}$; for $i=1, \ldots, q-1$ and $T_c(x) = (u_1 + c)x +u_2f(x); c \in \F$.
	
	Next, suppose that $H_i(x) \neq y.R(x)^{q^m-1}$ or $T_c(x) \neq R^{q^m}-R(x)$ for any $y \in \Fm^*$ and any $R(x) \in \mathbb{F}(x)$, where $\mathbb{F}$ is the algebraic closure of $\Fm$, then from Lemma \ref{2.1} and \ref{2.2} we have the following,
	
	\begin{equation}\label{cond3.3}
		|\chi_{f, a, b}(d_1, d_2, h_1, h_2)| \leq \M.(q-1).q^{\frac{m}{2}+1},
	\end{equation}
	
	where \[ \M=\begin{cases}
		2n_1+n_2,& \text{if } n_1>n_2,\\
		n_1+2n_2+1,& \text{if } n_1 \leq n_2. 
	\end{cases} \]
	
	\textbf{Claim:} %Let $q$ be a prime power, $m$ be a positive integer and $1 < n < q^m$ be a fixed integer. --------------------------------
	 If $H_i(x) = y.R(x)^{q^m-1}$ and $T_c(x)= R(x)^{q^m}-R(x)$ for some $y \in \Fm^*$ and some $R(x) \in \mathbb{F}(x)$, then we have $(d_1, d_2, h_1, h_2) = (0, 0, 0, 0)$.
	
	\textbf{Proof of the claim:} Suppose that $H_i(x) = y.R(x)^{q^m-1}$ for some $y \in \Fm^*$ and some $R(x) \in \mathbb{F}(x)$. Write $f=\frac{f_1}{f_2}$ and $R(x) = \frac{r_1}{r_2}$ where $f_1, f_2$ are co-prime polynomials over $\Fm$, $r_1, r_2$ are co-prime polynomials over $\mathbb{F}$ and $f_2, r_2$ are both monic. Then we have the following expression,
	
	\begin{equation}\label{cond3.4}
		x^{n_1+\frac{(q^m-1)i}{q-1}}.f_1^{n_2}.r_2^{q^m-1}=y.r_1^{q^m-1}.f_2^{n_2}
	\end{equation}
	
	Next write $f_1= \prod_{i=1}^{t}f_{1i}^{s_i}$ be the decomposition of $f_1$ into irreducible polynomials over $\Fm$, where $s_i$'s are positive integers. Since $\mathbb{F}$ is the algebraic closure of $\Fm$, $f_{1i}$ completely reduces into distinct linear factors over $\mathbb{F}$. From Equation \ref{cond3.4}, we have that, each $f_{1i}$ divides $r_1$ over $\mathbb{F}$, since $\gcd(f_1, f_2)=1$. Thus, $f_{1i}$ divides $r_1$ over $\mathbb{F}$ for all $i = 1, 2, \ldots, t$. By comparing the degrees of the factors $f_{1i}^{s_in_2}$ on both sides of the Equation \ref{cond3.4} over $\mathbb{F}$, we get $s_i.n_2=k_i.(q^m-1)$, where $k_i$ is the multiplicity of $f_{1i}$ in the factorization of $R(x)$. This implies that $s \mid s_i$, where $s=\frac{q^m-1}{\gcd(q^m-1, n_2)}$. Therefore, $f_1= (f'_1)^s$ for some polynomial $f'_1(x) \in \Fm[x]$. In a similar manner, $f_2= (f'_2)^s$ for some polynomial $f'_2(x) \in \Fm[x]$. Thus, we have that $f=\frac{f_1}{f_2}=(\frac{f'_1}{f'_2})^s$, where $s \mid (q^m-1)$ and $f \in \R_n$. This further implies that $s=1$, i.e., $\gcd(q^m-1, n_2)=q^m-1$, i.e., $(q^m-1)\mid n_2$, which implies $n_2=0$.
	
	Now, Equation \ref{cond3.4} becomes, 
	
	$$x^{n_1+\frac{(q^m-1)i}{q-1}}=y.R(x)^{q^m-1}$$
	
	From the above expression, it can be clearly seen that $R(x)= x^k$ for some positive integer $k$ and comparing the degrees on both sides we get
	
	$$n_1+\frac{(q^m-1)i}{q-1} = k(q^m-1)$$
	
	This implies that, 
	
	$$n_1 = (k-\frac{i}{q-1})(q^m-1) \geq (k-1)(q^m-1)$$
	
	which is impossible unless $k=1$, as $0 \leq n_1 \leq q^m-2$. Thus, $n_1 = (1- \frac{i}{q-1})(q^m-1)$. Since, $\chi_{d_1}=\chi_{q^m-1}^{n_1}$, there exists some positive integer $u$ such that $n_1=\frac{(q^m-1)u}{d_1}$. Substituting the value of $n_1$ we get    
	
	$$\frac{(q^m-1)u}{d_1}= (1- \frac{i}{q-1})(q^m-1)$$
	
	This implies $u(q-1)=(q-1-i)d_1$ which further implies $d_1 | u$, since $\gcd(d_1, q-1)=1$. Therefore, $q^m-1 \mid n_1$, which is possible only if $n_1=0$. Hence, $n_1=n_2=0$, that is $d_1=d_2=0$. 
	
	Secondly, for $c \in \F$, suppose $T_c(x)= R(x)^{q^m}-R(x)$ for some $R(x) \in \mathbb{F}(x)$. Write $f=\frac{f_1}{f_2}$ and $R=\frac{r_1}{r_2}$, where $f_1, f_2 \in \Fm[x]$ and $r_1, r_2 \in \mathbb{F}[x]$ , such that $\gcd(f_1, f_2)=1$ and $\gcd(r_1, r_2)=1$, then we have 
	
	\begin{equation}\label{cond3.5}
		r_2^{q^m}[f_2(u_1+c)+u_2f_1]=f_2[r_1^{q^m}-r_1.r_2^{q^m-1}]
	\end{equation}
	
	Suppose at least one of $u_1$ or $u_2$ is non-zero. If $u_2 \neq 0$, then $f_2=r_2^{q^m}$, since $\gcd(f_1, f_2)=1$ and $\gcd(r_1, r_2)=1$. From the second part of the definition of the set $\R_n$, we know that there exists a monic irreducible factor $h \in \Fm[x]$ of $f_2$ with multiplicity $t$ such that $q^m \nmid t$. By comparing degrees of $h$ on both sides of the expression $f_2=r_2^{q^m}$ over $\mathbb{F}$, we get that $t=s.q^m$, where $s$ is such that $g^s \mid \mid r_2$, that is, $s$ is the highest power of $g$ that divides $r_2$. This is clearly not possible since $q^m \nmid t$. Therefore $u_2=0$, and Equation \ref{cond3.5} becomes,
	
	$$ (u_1+c)x.r_2^{q^m}=r_1^{q^m}-r_1.r_2^{q^m-1}$$ 
	
	Since $\gcd(r_1, r_2) =1$, the above expression implies that $r_2$ is a non-zero constant polynomial and the above Equation becomes $Cx= r_1^{q^m}-r_1.r_2^{q^m-1}$, where $C=(u_1+c).r_2^{q^m} \in \mathbb{F}$. Therefore, this implies that $(u_1+c)=0$, since else $\deg(r_1).q^m=1$. We observe that $\psi_{g_1}(g_1(\alpha))=1$ which implies $\widehat{\psi_1}(u_1.g_1(\alpha))=1$. Thus, if $u_1+c=0$, then $\widehat{\psi_1}(u_1.g_1(\alpha))=1$, which implies $\widehat{\psi_1}((-c).g_1(\alpha))=1$. Since $c\in \F$ and $x^m-1$ is the $\F$-order of $\widehat{\psi_1}$, we have $(x^m-1) \mid c.g_1$, which further implies $c=0$, that is $u_1=0$. Hence, $u_1=u_2=0$ that is $h_1=h_2=0$. 
	
	The following lemma deals with the case when $(d_1, d_2, h_1, h_2)=(0, 0, 0, 0)$. 
	
	\begin{lemma}\label{3.2}
		If $(d_1, d_2, h_1, h_2) = (0, 0, 0, 0)$, then 
		
		$$ |\chi_{f, a, b}(0, 0, 0, 0)-(q^m-|P|)| \leq (q-2)(|P|-1)+(q-1)^2.q^{\frac{m}{2}}$$
	\end{lemma}
	
	\begin{proof}
		The proof is similar to Lemma 3.3 of \cite{avnish2}, hence omitted. 
	\end{proof}
	
	Now, from Equation \ref{cond3.3} and the claim above we have that,
	
	$$|\chi_{f, a, b}(d_1, d_2, h_1, h_2)| \leq \M.(q-1).q^{\frac{m}{2}+1}$$
	
	unless $(d_1, d_2, h_1, h_2) = (0, 0, 0, 0)$. Furthermore, from Lemma \ref{3.2} we have 
	
	\begin{align*}
		 |\chi_{f, a, b}(0, 0, 0, 0)-(q^m-|P|)| \leq & ~(q-2)(|P|-1)+(q-1)^2.q^{\frac{m}{2}}\\
		 \geq & ~(q^m-|P|)-(q-2)(|P|-1)-(q-1)^2.q^{\frac{m}{2}}\\
		 \geq & ~q^m-(q-1)(|P|-1)-(q-1)^2.q^{\frac{m}{2}}\\
		 \geq & ~q^m-(q-1).n-(q-1)^2.q^{\frac{m}{2}}.
	\end{align*}
	
	Therefore, from Equation \ref{cond3.1}, we get 
	
	\begin{align*}
		\Nm_{f, a, b}(Q_{e_1}, e_2, R_{g_1}, g_2) >& \vt \{q^m-(q-1)\{n+(q-1)q^{\frac{m}{2}}+\M.q^{\frac{m}{2}+1}(W(Q_{e_1}).W(e_2).W(R_{g_1})W(g_2)-1))\}\} \\
		>& \vt\{q^m-\M.q^{\frac{m}{2}+1}(q-1)W(Q_{e_1}).W(e_2).W(R_{g_1})W(g_2)\} \\
		>& \vt\{q^m-\M.q^{\frac{m}{2}+2}W(Q_{e_1}).W(e_2).W(R_{g_1})W(g_2)\}
	\end{align*} 
	
	Hence $\Nm_{f, a, b}(Q_{e_1}, e_2, R_{g_1}, g_2)>0$ provided $q^{\frac{m}{2}-2}>\M.W(Q_{e_1}).W(e_2).W(R_{g_1})W(g_2)$, where \[ \M=\begin{cases}
		2n_1+n_2,& \text{if } n_1>n_2,\\
		n_1+2n_2+1,& \text{if } n_1 \leq n_2. 
	\end{cases} \]

	Summarizing the discussion above we arrive at the following result.
	
	\begin{theorem}\label{suffcon}
		Let $q$ be a prime power and $m, n \in \mathbb{N}$ such that $m\geq5$. Suppose $e_1, e_2$ divide $q^m-1$ and $g_1,g_2$ divide $x^m-1$. Then we have
		\begin{equation*}
			\Nm_{f, a, b}(Q_{e_1}, e_2, R_{g_1}, g_2) > \vt\{q^m-\M.q^{\frac{m}{2}+2}W(Q_{e_1}).W(e_2).W(R_{g_1})W(g_2).
		\end{equation*}
		In particular, we have $(q, m) \in \Q(n_1, n_2)$ provided that 
		\begin{equation}\label{cond3.6}
			q^{\frac{m}{2}-2}>\M.W(Q_{e_1}).W(e_2).W(R_{g_1})W(g_2),
		\end{equation}
		where \[ \M=\begin{cases}
			2n_1+n_2,& \text{if } n_1>n_2,\\
			n_1+2n_2+1,& \text{if } n_1 \leq n_2. 
		\end{cases} \]
	\end{theorem}
	
	It is clear from the above discussion that $(q,m) \in \Q(n_1, n_2)$, when $e_1=e_2=q^m-1$ and $g_1=g_2=x^m-1$, that is provided 
	
	\begin{equation}\label{cond3.7}
		q^{\frac{m}{2}-2} > \M.W(Q_{q^m-1}).W(q^m-1)W(R_{x^m-1}).W(x^m-1),
	\end{equation}where \[ \M=\begin{cases}
		2n_1+n_2,& \text{if } n_1>n_2,\\
		n_1+2n_2+1,& \text{if } n_1 \leq n_2. 
	\end{cases} \]
	
	Here $Q_{q^m-1}$ is the largest divisor of $q^m-1$ such that $\gcd(Q_{q^m-1}, q-1)=1$. In particular, $Q_{q^m-1}= \frac{q^m-1}{(q-1)\gcd(m, q-1)}$ and $W(Q_{q^m-1})\leq W(q^m-1)$. Similarly, $R_{x^m-1}$ is the largest divisor of $x^m-1$ such that $\gcd(R_{x^m-1}, x-1)=1$. Furthermore, $W(R_{x^m-1})=W(x^m-1)/2$.

	\section{Sieving Inequality}
	
	For our problem we use the "additive-multiplicative" sieve that involves sieving with respect to the primes in $q^m-1$ as well as the irreducible factors in $x^m-1$. We apply the sieving inequality established by Kapetanakis in \cite{kapetanakis2014normal} to give a modified form of (\ref{cond3.6}), in order to study the pairs $(q, m)$ which do not satisfy the condition (\ref{cond3.6}).  
	
	\begin{lemma}[Sieving Inequality] \label{sieveineq}
		
		Let $d'$ be a divisor of $Q_{q^m-1}$ and $p'_1, p'_2,\dots , p'_r$ be the remaining distinct primes dividing $Q_{q^m-1}$; let $d$ be a divisor of $q^m-1$ and $p_1, p_2,\dots , p_s$ be the remaining distinct primes dividing $q^m-1$. Furthermore, let $g'$ be a divisor of $R_{x^m-1}$ and $g'_1, g'_2, \dots , g'_t$ be the remaining distinct irreducible factors of $R_{x^m-1}$; let $g$ be a divisor of $x^m-1$ and $g_1, g_2, \dots , g_u$ be the remaining distinct irreducible factors of $x^m-1$. Abbreviate  $ \mathfrak{N}_{f, a, b}(Q_{q^m-1}, q^m-1, R_{x^m-1}, x^m-1) $ to $\mathfrak{N}_{f,a,b}$. Then
		\begin{multline} \label{cond4.1}
			\mathfrak{N}_{f,a,b} \geq \sum_{i=1}^r\mathfrak{N}_{f,a,b}(p'_i d', d, g', g)+  \sum_{i=1}^s\mathfrak{N}_{f,a,b}( d',p_i d, g', g)+ \sum_{i=1}^t\mathfrak{N}_{f,a,b}(d', d,g'_i g', g)\\ 
			+  \sum_{i=1}^u\mathfrak{N}_{f,a,b}(d', d, g',g_i g) -(r+s+t+u-1)\mathfrak{N}_{f,a,b}(d', d, g', g).
		\end{multline}
	\end{lemma}
	
	\begin{theorem}\label{thsieve}
		Let $m, q \in \mathbb{N}$ be such that q is a prime power, $m \geq 5$. Define 
		\[ \lambda := 1 - \sum_{i=1}^r\frac{1}{p'_i} - \sum_{i=1}^s\frac{1}{p_i} - \sum_{i=1}^t \frac{1}{q^{{\mathrm deg}(g'_i)}} - \sum_{i=1}^u \frac{1}{q^{{\mathrm deg}(g_i)}}, \  \] and \[ \Lambda  := \frac{r+s+t+u-1}{\lambda}+2. \]
		Suppose $\lambda >0$, then $\mathfrak{N}_{f,a,b}>0$ if 
		\begin{equation}\label{cond4.2}
			q^{\frac{m}{2}-2}> \M.W(d')W(d)W(g')W(g)\Lambda,
		\end{equation}
		where \[ \M=\begin{cases}
			2n_1+n_2,& \text{if } n_1>n_2,\\
			n_1+2n_2+1,& \text{if } n_1 \leq n_2. 
		\end{cases} \]
	\end{theorem}
	
	\begin{proof}
		The proof is similar to Proposition 5.3 of \cite{kapetanakis2014normal}, hence omitted. 
	\end{proof}
	
	\section{Computational Results}
	
	The above discussed results hold for any finite field $\Fm$ of any prime characteristic and for any arbitrary natural number $n$. In order to make computations less tedious and illustrate how the aforementioned results can be applied, we assume that $q=7^k$ where $k$ is a positive integer and $(n_1, n_2)=(1, 1)$, that is, $n=n_1+n_2=2$. 
	
	We use the concept of the $7$-free part of $m$ i.e. $m'$, where $m^\prime$ is such that $m= 7^k m^\prime$, such that $\gcd(7,m^\prime)=1$ and $k$ is a non-negative integer. We then have that $\omega(x^m-1)=\omega(x^{m^\prime}-1)$ which implies $W(x^m-1)=W(x^{m^\prime}-1)$.
	
	We shall need the following lemma which can be derived from Lemma ~6.2 of Cohen's work \cite{cohenlemmas}. 
	
	\begin{lemma}\label{mbound}
		Let $M$ be a positive integer, then 
		\begin{enumerate}[label=(\roman*)]
			\item $W(M) < 244.66 \times M^{1/7}$.
			\item $W(M) < 1.11 \times 10^9 \times M^{1/10}$.
			\item $W(M) < 5.09811 \times 10^{67} \times M^{1/14}$.
		\end{enumerate} 
	\end{lemma}
	
	When $q=7^k$ and $m=6$, we observe that $x^6-1$ factors into a product of six linear factors. Then using Lemma \ref{mbound}, $(q, m)$ $\in$ $Q(1, 1)$ if $q^{1/7}>4\cdot2^{11}\cdot(5.09811\times10^{67})^2$,  which implies $q> 1.984 \times 10^{975}$, that is the case $k \geq 1155$ is settled. Similarly, when $m=5$, we get that $(q, m)$ $\in$ $Q(1, 1)$ if $q>1.026\times10^{269771}$ which implies $k \geq 319219$, is hence settled. To treat the remaining cases it requires huge computational resources. Thus, moving forward, we shall assume that $m \geq 7$.

	Next we divide our discussion into cases depending on the nature of factors of $x^{m ^ \prime}-1$. 
	
	\subsection{Values of $m$ such that $m'$ divides $q-1$.}\label{5.1}
	
	When $m'\mid q-1$, then $x^{m'}-1$ splits into a product of $m'$ linear factors over $\F$. Let $d'= Q_{q^m-1}$, $d= q^m-1$, $g'=1$ and $g=1$ in Theorem \ref{thsieve}, then $\Lambda = \frac{2q^2+2a-6q+4}{(a-2)q+a+2}$, where $a=\frac{q-1}{m'}$. In particular, $\Lambda < 2q^2$. Thus, $(q,m) \in \Q(1, 1)$ if $q^{\frac{m}{2}-4} > 8.W(q^m-1)^2$. Applying Lemma \ref{mbound}, it is sufficient if $q^{\frac{3m}{14}-4}>8.(244.66)^2$. This inequality holds for $q \geq 49$ and for all $m \geq 35$. In particular, for $q \geq 49$ and $m' \geq 35$. 
	
	Next,  we  investigate all   cases with $m^\prime \leq 34$. For this, we set $d'= Q_{q^m-1}$, $d= q^m-1$, $g'=1$ and $g=1$ in Theorem \ref{thsieve} unless otherwise mentioned. Here $\lambda = 1- \frac{2m'-1}{q}$ and $\Lambda = 2+\frac{2q(m'-1)}{q-2m'+1}$. 
	
	\subsubsection{\underline{$m ^\prime =1$}}  
	
	In this case $m=7^j$ for some positive integer $j$ and $\Lambda = 2$. Then by Theorem \ref{thsieve}, it is sufficient if $q^{\frac{m}{2}-2} > 8\cdot W(q^m-1)^2$. Using Lemma \ref{mbound}, the above becomes $q^{3\cdot7^j/10-2}>8\cdot (1.11\times10^9)^2$. This holds for $q=7$ for all $j \geq 3$. For $7^2 \leq q \leq 7^{224}$ for all $j \geq 2$ and for $q \geq 7^{225}$ for all $j \geq 1$. For the above $225$ pairs we checked $q^{\frac{m}{2}-2}>8\cdot W(q^m-1)^2$ by factoring $q^m-1$ directly and get that all except the pairs $(7, 7), (7^2, 7), (7^3, 7)$ and $(7^4, 7)$ satisfy the condition.
	
	\subsubsection{\underline{$m ^\prime =2$}}
	
	Here $m=2 \cdot 7^j$ for some positive integer $j$ and $\Lambda = 2 + \frac{2q}{q-3} < 6$. Now by Theorem \ref{thsieve} and Lemma \ref{mbound}, it suffices to see that $q^{3\cdot(2\cdot7^j)/14-2}>24\cdot (244.66)^2$. The above inequality holds for $7 \geq q \geq 7^8$ for all $j \geq 2$ and for $q \geq 7^8$; this holds for all $j \geq 1$. Proceeding in a similar way and verifying the condition  $q^{\frac{m}{2}-2}>24 \cdot W(q^m-1)^2$ for the above seven pairs, we get that $(7, 14)$ fails to satisfy the condition.
	
	\subsubsection{\underline{$3 \leq m ^\prime \leq 34$}}
	
	Proceeding in a similar fashion for the cases $3 \leq m ^\prime \leq 34$, we find that there is no exceptional pair for the aforementioned cases. 
	
	Now, for the above possible exceptional pairs we choose suitable values of $d'$, $d$, $g'$ and $g$ such that all the pairs satisfy Theorem~\ref{thsieve} (see Table~\ref{table1}) except the pair $(7, 7)$. Therefore, there is one exceptional pair for the case $m^\prime \mid  q-1$.
	
	\begin{table}[ht]
		\begin{center}\small
			\begin{tabular}{c c c c c c c c }
				\hline 
				\# & $(q,m)$ & $d'$ & $d$ & $g'$ & $g$ & $\lambda>$ & $\Lambda<$  \\
				\hline
				1 & $(7^2, 7)$ & 29 & 2 & 1 & 1 & 0.591458675085391 & 15.5258815822509  \\
				2 & $(7^3, 7)$ & 1 & 2 & 1 & 1 & 0.541731553445425 & 16.7674617605709  \\
				3 & $(7^4, 7)$ & 1 & 2 & 1 & 1 & 0.376967586454699 & 33.8329756488017  \\
				4 & $(7, 14)$ & 2 & 2 & $(x+1)$ & $(x-1)$ & 0.434526936872864 & 22.7121796976956  \\
				\hline
			\end{tabular}
			\caption{Pairs $(q, m)$ with $m^\prime \mid  q-1$ for which Theorem \ref{thsieve} holds for the above choices of $d'$, $d$, $g'$ and $g$.\label{table1}}
		\end{center}
	\end{table}

	\subsection{Values of $m$ such that $m'$ does not divide $q-1$.}\label{5.2}
	
	Given $q$ and $m$, let $u$ be the order of $q$ mod $m^\prime$. Then $x^{m^\prime}-1$ is a product of irreducible polynomial factors of degree less than or equal to $u$ in $\mathbb{F}_q[x]$; in particular, $u \geq 2$ if $m^\prime \nmid q-1$. Let $M$ be the number of distinct  irreducible factors of $x^m-1$ over $\mathbb{F}_q$ of degree less than $u$. Let $\sigma(q,m)$ denotes the ratio $$ \sigma(q,m):= \frac{M}{m},$$ where $m\sigma(q,m)= m^\prime\sigma(q,m^\prime)$.
	
	We need the following bounds which can be deduced from Proposition~5.3 of \cite{cohen1}.
	
	\begin{lemma}\label{sibd}
		Suppose $q=7^k$, $m^\prime>4$ and $m_1=\gcd(q-1, m^\prime)$
		\begin{enumerate}[label=(\roman*)]
			\item If $m^\prime=2m_1$ and $q$ odd, then $u=2$ and $\sigma(q, m^\prime)=1/2$. 
			\item If $m^\prime=4m_1$ and $q \equiv 1 \pmod{4}$, then $u=4$ and  $\sigma(q, m^\prime)=3/8$. 
			\item If $m^\prime=6m_1$ and $q \equiv 1 \pmod{6}$, then $u=6$ and  $\sigma(q, m^\prime)=13/36$. 
			\item Otherwise, $\sigma(q, m^\prime)\leq 1/3$. 
		\end{enumerate}
	\end{lemma}
	
		We need the following Lemma 5.3 \cite{rani1} which we adjusted properly. It provides a bound on $\Lambda$, for suitable development of the sufficient condition.
	
	\begin{lemma}\label{lambd}
		Assume that $q=p^k$ and $m$ is a positive integer such that $m^\prime\nmid q-1$.  Let $u (>2)$ denote the order of $q \mod m'$.   Let $g$ be the product of the irreducible factors of $x^{m'}-1$ of degree less than $eu$.  Then, in the notation of Theorem \ref{thsieve} with $d'=Q_{q^m-1}$, $d=q^m-1$ and $g'=g/(x-1)$, we have $\Lambda \leq 2m^\prime$.
	\end{lemma}
	
	We now observe that $q= 1, 2, 3$ and $6$ divide $q-1$ for any $q= 7^k$ and it has been discussed earlier, whereas $m'=7$ is not possible. Therefore, it suffices to discuss $m'=4, 5$ and $m^\prime \geq 8$ for which $m^\prime \nmid q-1$. 
	
	We first deal with the case $m'=4$. Then $m= 4\cdot 7^j$ for some positive integer $j$. When $q=7^k$ and $m' \nmid q-1$, we note that $k$ is odd and in this case $x^{m'}-1$ is the product of two linear factors and a quadratic factor. This implies that $W(x^m-1)=W(x^{m'}-1)=2^3=8$ and $W(R_{x^m-1})=2^2=4$. Thus, $(q, m) \in \Q(1, 1)$ if $q^{\frac{m}{2}-2} > 128 \cdot W(q^m-1)^2$. Using Lemma \ref{mbound} it is sufficient that $q^{\frac{3m}{14}-2} > 128 \cdot (244.66)^2$. The above inequality holds for $q=7$, $m \geq 48$ and for $q \geq 7^3$, $m \geq 23$. Now for the only possible exceptional pair $(7, 28)$, we check the condition $q^{\frac{m}{2}-2}> 128 \cdot W(q^m-1)^2$ and found that it satisfies the condition.
	
	We treat the case $m'=5$ in two sub-cases. When $q=7^k$ with $k$ odd, then $x^{m'}-1$ is a product of a linear factor and a factor of degree four, which implies that $W(x^m-1)=W(x^{m'}-1)=2^2=4$ and $W(R_{x^m-1})=2$. Hence from Theorem \ref{suffcon} and Lemma \ref{mbound} it is sufficient that $q^{\frac{3m}{14}-2}> 32 \cdot (244.66)^2$. The above holds for $ q=7$, $m \geq 45$ and for $q \geq 7^3$, $m \geq 21$. For the possible exceptional pair $(7, 35)$, we check $ q^{\frac{m}{2}-2} > 32\cdot W(q^m-1)^2$ and got it verified to be true. Again when $q=7^k$, where $k$ is even and $k \not\equiv 0 \pmod{4}$, $x^{m'}-1$ is a product of a linear factor and two quadratic factors. In this case, $W(x^m-1)=W(x^{m'}-1)=2^3=8$ and $W(R_{x^m-1})=2^2=4$. Again applying Theorem \ref{suffcon} and Lemma \ref{mbound}, it is sufficient that $q^{\frac{3m}{14}-2}> 128 \cdot (244.66)^2$, which holds for $q \geq 7^2$ and $m \geq 29$. Thus there is no exceptional pair in this case.
	
	Next suppose $m^\prime \geq 7$. Let $d'=Q_{q^m-1}$, $d=q^m-1$, $g'=g/(x-1)$ and $g$ be the product of the irreducible factors of $x^{m'}-1$ of degree less than $u$. Thus, Theorem \ref{thsieve} along with Lemma \ref{lambd} give us the following sufficient condition $$q^{\frac{m}{2}-2}>8\cdot m^\prime\cdot W(q^m-1)^2\cdot2^{2m^\prime\sigma(q, m^\prime)},$$ that is, it suffices that, 
	\begin{equation}\label{cond5.1}
		q^{\frac{m}{2}-2}>8\cdot m\cdot W(q^m-1)^2\cdot2^{2m\sigma(q, m^\prime)} .
	\end{equation}
	
	\begin{lemma}
		Let $q=7$ and $m$ be a positive integer such that $m^\prime\nmid q-1$. Then for $m^\prime\geq 8$ we have that, for any prescribed primitive norm $a$ and a non-zero trace $b$ over $\F$, there exists an element $\alpha \in \Fm$ with $\N_{{\Fm}/{\F}}(\alpha)= a$ and $\Tr_{{\Fm}/{\F}}(\alpha)= b$, for which $(\alpha, f(\alpha))$ is  a primitive normal pair  in $\mathbb{F}_{q^m}$ over $\mathbb{F}_{q}$ except, possibly the pairs $(7, 8)$, $(7, 9)$, $(7, 10)$, $(7, 12)$ and  $(7, 18)$.
	\end{lemma}
	
	\begin{proof}
		We divide our discussion into the following two cases.
		
		\textbf{Case 1.} When $m^\prime \neq 2m_1, 4m_1, 6m_1$, we have from Lemma \ref{sibd} that $\sigma(q, m^\prime)\leq 1/3$. Substituting this in (\ref{cond5.1}) and also using Lemma \ref{mbound}, we have the sufficient condition 
		$$q^{\frac{3m}{10}-2}>8\cdot m\cdot (1.11\times10^9)^2\cdot2^{\frac{2m}{3}-1}$$
		The above inequality holds for $m\geq 436$. Next for $m \leq 435$, we verify the condition (\ref{cond3.7}) and found that $(7, m) \in \Q(1, 1)$ except when $m$= $8$, $9$, $10$, $11$, $12$, $15$, $16$, $18$, $19$, $20$, $24$, $27$, $30$, $32$ and $48$. We next choose appropriate values of $d'$, $d$, $g'$ and $g$ and get that the pairs $(7, m)$ where $m$= $11$, $15$, $16$, $19$, $20$, $24$, $27$, $30$, $32$ and $48$ satisfy Theorem~\ref{thsieve} (see Table~\ref{table2}). Therefore, in this case the possible exceptional pairs are  $(7, 8)$, $(7, 9)$, $(7, 10)$, $(7, 12)$ and  $(7, 18)$.

		\textbf{Case 2.} When $m^\prime = 6m_1$, the only possible value of $m'$ is $36$ and in this case $x^{m'}-1$ is a product of $15$ irreducible factors. Proceeding in a similar way we find that it is sufficient that $q^{\frac{3m}{14}-2}> 2^{31} \cdot (244.66)^2$, which holds for $m \geq 88$. Again for the only possible exceptional pair $(7, 36)$, we checked the sufficient condition (\ref{cond3.7}) and get that it fails to satisfy it but it satisfies Theorem \ref{thsieve} for appropriate values of $d'$, $d$, $g'$ and $g$(see Table~\ref{table2}). Hence in this case there is no exceptional pair. 
		
		The proof is now complete.
	\end{proof}
	
	\begin{table}[ht]
		\begin{center}\small
			\begin{tabular}{c c c c c c c c }
				\hline 
				\# & $(q,m)$ & $d'$ & $d$ & $g'$ & $g$ & $\lambda>$ & $\Lambda<$  \\
				\hline

				1 & $(7, 11)$ & 1 & 2 & $\frac{x^{11}-1}{x+6}$ & $x+6$ & 0.664878903964168 & 9.52016640953532  \\				
				2 & $(7, 15)$ & 1 & 2 & $x+3$ & $x+5$ & 0.0650904516114856 & 263.175019977896  \\
				3 & $(7, 16)$ & 2 & 6 & $x+1$ & $x+6$ & 0.0521044379491288 & 424.228909204995  \\
				4 & $(7, 19)$ & 1 & 2 & $\frac{x^{19}-1}{x-1}$ & $x+6$ & 0.644400685606226 & 17.5182950970830  \\
				5 & $(7, 20)$ & 2 & 2 & $x+1$ & $x^2+6$ & 0.0219001519714673 & 1006.55923907116  \\
				6 & $(7, 24)$ & 10 & 6 & $\frac{x^6-1}{x-1}$ & $x^6+6$ &  0.0716671887608818 & 476.415148519935  \\
				7 & $(7, 27)$ & 3 & 2 & $x+3$ & $x+6$ &  0.0435777658630945 & 575.687051294481  \\
				8 & $(7, 30)$ & 22 & 66 & $x^3+1$ & $x^3+1$ &  0.0532274038983268 & 565.619447931464  \\
			    9 & $(7, 32)$ & 2 & 6 & $x+1$ & $x+6$ &  0.0431067705041303 & 790.739207376768  \\	
			    10 & $(7, 48)$ & 131806610 & 30 & $\frac{x^{12}-1}{x-1}$ & $x^{12}+6$ &  0.0431067705041303 & 790.739207376768  \\	
			    11 & $(7, 36)$ & 30 & 30 & $x^4+x^2+1$ & $x^4+x^2+1$ &  0.0529987428084201 & 681.261395503905  \\	
			    12 & $(7^2, 9)$ & 3 & 6 & $1$ & $x+6$ &  0.710605990818367 & 25.9232432876368  \\	
			    13 & $(7^2, 10)$ & 1 & 6 & $x+1$ & $x+6$ &  0.375641648570107 & 55.2422325270126  \\
				14 & $(7^2, 15)$ & 1 & 6 & $\frac{x^{15}-1}{x-1}$ & $x+6$ &  0.547377568602427 & 44.0185285610515 \\
				15 & $(7^2, 16)$ & 5 & 6 & $x+1$ & $x+6$ &  0.0831734094045884 & 470.899859693000  \\
				16 & $(7^2, 18)$ & 15 & 30 & $\frac{x^{18}-1}{x-1}$ & $x^3+x^2+4x+1$ & 0.566118276336117 & 40.8611371856472  \\
				17 & $(7^2, 20)$ & 5 & 2 & $x+1$ & $x+6$ &  0.106897950411074 & 357.479219703201  \\
				18 & $(7^3, 8)$ & 2 & 6 & $x+1$ & $x+6$ &  0.288679255430203 & 81.6731998138361  \\
				19 & $(7^3, 10)$ & 2 & 6 & $x+1$ & $x+6$ &  0.693040559610923 & 29.4154228587526  \\
				20 & $(7^3, 18)$ & 2 & 2 & $x+1$ & $x+6$ &  0.393704090361428 & 134.078891921755  \\
				21 & $(7^4, 9)$ & 3 & 6 & $x+3$ & $x+6$ & 0.426127285898464 & 55.9744831207086  \\
				22 & $(7^5, 8)$ & 2 & 2 & $x+1$ & $x+6$ &  0.0159001518760136 & 1511.41954436332  \\
				23 & $(7^6, 8)$ & 5 & 2 & $x+1$ & $x+6$ & 0.0637988343818541 & 534.925096977484  \\
				24 & $(7^3, 12)$ & 10 & 2 & $x+1$ & $x+6$ &   0.120384944971172 & 259.507282222236  \\

				\hline
			\end{tabular}
			\caption{Pairs $(q, m)$ with $m^\prime \nmid  q-1$ for which Theorem \ref{thsieve} holds for the above choices of $d'$, $d$, $g'$ and $g$.\label{table2}} 
		\end{center}
	\end{table}

	\begin{lemma}
		Let $q\geq7^2$ and $m$ be a positive integer such that $m^\prime\nmid q-1$. Then for $m^\prime\geq 8$ we have that, for any prescribed primitive norm $a$ and a non-zero trace $b$ over $\F$, there exists an element $\alpha \in \Fm$ with $\N_{{\Fm}/{\F}}(\alpha)= a$ and $\Tr_{{\Fm}/{\F}}(\alpha)= b$, for which $(\alpha, f(\alpha))$ is  a primitive normal pair  in $\mathbb{F}_{q^m}$ over $\mathbb{F}_{q}$.
	\end{lemma}
	
	\begin{proof}
		We shall consider the following four cases. 
		
		\textbf{Case 1. $m^\prime \neq 2m_1, 4m_1, 6m_1$.} In this case Lemma \ref{sibd} implies $\sigma(q, m^\prime)\leq 1/3$. Applying Lemma \ref{mbound}, we get $(q, m) \in \Q(1, 1)$ whenever $q^{\frac{3m}{14}-2}>8\cdot m\cdot (244.66)^2\cdot2^{\frac{2m}{3}-1}$. This holds for $q=7^2$ for all $m\geq 55$. Next for $q\geq 7^2$ and $q^m < (7^2)^{55}$, we checked the condition (\ref{cond3.7}) and get that all pairs except $(7^2, 9)$, $(7^2, 10)$, $(7^2, 15)$, $(7^2, 16)$, $(7^2, 18)$, $(7^2, 20)$, $(7^3, 8)$, $(7^3, 10)$, $(7^3, 18)$, $(7^4, 9)$, $(7^5, 8)$ and $(7^6, 8)$ satisfy the condition. 
		
	    \textbf{Case 2. $m^\prime = 2m_1$.} From Lemma \ref{sibd}, $\sigma(q, m^\prime)\leq 1/2$. Thus, using Lemma \ref{mbound} condition (\ref{cond5.1}) becomes,  $q^{\frac{3m}{14}-2}>8\cdot m\cdot (244.66)^2\cdot2^{m-1}$. For $q=7^2$ this holds for all $m \geq 152$. Next for the remaining pairs, we verified the sufficient condition (\ref{cond3.7}) and found that the pair $(7^3, 12)$ fails to satisfy it. 
	    
	    \textbf{Case 3. $m^\prime = 4m_1$.} Here, $\sigma(q, m^\prime)\leq 3/8$. Similarly, the condition (\ref{cond5.1}) becomes $q^{\frac{3m}{14}-2}>8\cdot m\cdot (244.66)^2\cdot2^{\frac{3m}{4}-1}$ which holds for $q=7^2$ for all $m \geq 66$. For $q \geq 7^2$ and $q^m \leq 49^{66}$, upon verification we get that condition (\ref{cond3.7}) holds for all pairs. Thus, in this case there is no exceptional pair.
	    
	    \textbf{Case 4. $m^\prime = 6m_1$.} In this case $\sigma(q, m^\prime)\leq 13/36$. Then the condition (\ref{cond5.1}) becomes $q^{\frac{3m}{14}-2}>8\cdot m\cdot (244.66)^2\cdot2^{\frac{13m}{18}-1}$, which holds for $q=7^2$ for all $m \geq 62$. For $q \geq 7^2$ and $q^m \leq 49^{62}$ and upon further verification we get that there is no exceptional pairs.
	    
	    Finally, we refer to Table \ref{table2}, to note that Theorem \ref{thsieve} holds for all the possible exceptional pairs in the above discussion for the suitable choices of $d'$, $d$, $g'$ and $g$. Thus, there is no possible exception in this case.
	    
	    This completes our proof.
	\end{proof}
	
	As an immediate consequence the discussion in Subsections \ref{5.1} and \ref{5.2} we obtain the Theorem \ref{comp}.

	\section{Declarations}
	\textbf{Conflict of interest} The authors declare no competing interests.
	
	\textbf{Ethical Approval} Not applicable.
	
	\textbf{Funding} The first author is supported by DST INSPIRE Fellowship(IF210206).
	
	\textbf{Data availability} Not applicable.
	
	\bibliographystyle{plain}
	\bibliography{ref3}

\begin{thebibliography}{10}

\bibitem{chou}
Wun-Seng Chou and Stephen~D Cohen.
\newblock Primitive elements with zero traces.
\newblock {\em Finite Fields and Their Applications}, 7(1):125--141, 2001.

\bibitem{cohenlemmas}
Stephen~D Cohen.
\newblock Pairs of primitive elements in fields of even order.
\newblock {\em Finite Fields and Their Applications}, 28:22--42, 2014.

\bibitem{cohach1}
Stephen~D Cohen and Dirk Hachenberger.
\newblock Primitive normal bases with prescribed trace.
\newblock {\em Applicable Algebra in Engineering, Communication and Computing},
  9:383--403, 1999.

\bibitem{cohach2}
Stephen~D Cohen and Dirk Hachenberger.
\newblock Primitivity, freeness, norm and trace.
\newblock {\em Discrete Mathematics}, 214(1-3):135--144, 2000.

\bibitem{cohen1}
Stephen~D Cohen and Sophie Huczynska.
\newblock The primitive normal basis theorem--without a computer.
\newblock {\em Journal of the London Mathematical Society}, 67(1):41--56, 2003.

\bibitem{fu}
Lei Fu and Daqing Wan.
\newblock A class of incomplete character sums.
\newblock {\em arXiv preprint arXiv:1303.3650}, 2013.

\bibitem{costas}
Solomon~W Golomb.
\newblock Algebraic constructions for costas arrays.
\newblock {\em Journal of Combinatorial Theory, Series A}, 37(1):13--21, 1984.

\bibitem{gupta1}
Anju Gupta, Rajendra~K Sharma, and Stephen~D Cohen.
\newblock Primitive element pairs with one prescribed trace over a finite
  field.
\newblock {\em Finite Fields and Their Applications}, 54:1--14, 2018.

\bibitem{kapetanakis2014normal}
Giorgos Kapetanakis.
\newblock Normal bases and primitive elements over finite fields.
\newblock {\em Finite Fields and Their Applications}, 26:123--143, 2014.

\bibitem{lens}
Hendrik~Willem Lenstra and RJ~Schoof.
\newblock Primitive normal bases for finite fields.
\newblock {\em Mathematics of Computation}, 48(177):217--231, 1987.

\bibitem{difhell}
Christof Paar, Jan Pelzl, Christof Paar, and Jan Pelzl.
\newblock Public-key cryptosystems based on the discrete logarithm problem.
\newblock {\em Understanding Cryptography: A Textbook for Students and
  Practitioners}, pages 205--238, 2010.

\bibitem{rani1}
Mamta Rani, Avnish~K Sharma, Sharwan~K Tiwari, and Indivar Gupta.
\newblock On the existence of pairs of primitive normal elements over finite
  fields.
\newblock {\em Sao Paulo Journal of Mathematical Sciences}, 16(2):1032--1049,
  2022.

\bibitem{sagesage}
SageMath Sage~Developers.
\newblock the sage mathematics software system (version 9.0), 2020.

\bibitem{avnish2}
Avnish~K Sharma, Mamta Rani, and Sharwan~K Tiwari.
\newblock Primitive normal pairs with prescribed norm and trace.
\newblock {\em Finite Fields and Their Applications}, 78:101976, 2022.

\end{thebibliography}
	
\end{document}